\documentclass[11pt,reqno]{amsart}
\usepackage{a4wide}
\usepackage{bm}
\usepackage{stmaryrd}
\usepackage{graphicx}
\usepackage{caption}
\usepackage{dsfont}

\newtheorem{thm}{Theorem}
\newtheorem{lem}{Lemma}
\newtheorem{prop}{Proposition}

\newtheorem{defi}{Definition}

\theoremstyle{definition}						

\newcommand{\R}{\mathbb{R}}

\newcommand{\N}{\mathbb{N}}
\newcommand{\Z}{\mathbb{Z}}
\newcommand{\Q}{\mathbb{Q}}
\newcommand{\Td}{\mathbb{T}^d}
\newcommand{\F}{\mathfrak{F}}
\newcommand{\V}{\mathcal{V}}
\newcommand{\e}{\epsilon}
\newcommand{\dist}{\textrm{dist}}

\newcommand{\ex}{\textrm{e}}
\newcommand{\Sph}{\mathbb{S}}
\newcommand{\bv}{\bm{v}}
\newcommand{\y}{\bm{y}}
\newcommand{\x}{\bm{x}}
\newcommand{\z}{\bm{z}}
\newcommand{\q}{\bm{q}}
\newcommand{\m}{\bm{m}}
\newcommand{\und}{\mathds{1}_\textrm{d}}

\begin{document}


\title{How far can you see in a forest?}
\author{Faustin ADICEAM}
\address{ Department of Mathematics, University of York,  York, YO10
5DD, UK } \email{faustin.adiceam@york.ac.uk}
\thanks{The author's research was partly supported by EPSRC Programme Grant EP/J018260/1. He would like to thank Professor Barak Weiss for suggesting the problem and the referees for valuable comments which improved an earlier draft of the paper.}

\begin{abstract}
We address a visibility problem posed by Solomon \& Weiss~\cite{solweiss}. More precisely,  in any dimension $n:=d+1\ge 2$, we construct a forest $\F$ with finite density satisfying the following condition~: if $\e>0$ denotes the radius common to all the trees in $\F$, then the visibility $\V$ therein satisfies the estimate $\V(\e)\, =\, O_{\eta, d}\left(\e^{-2d-\eta}\right)$ for any $\eta>0$, no matter where we stand and what direction we look in. The proof involves Fourier analysis and sharp estimates of exponential sums.
\end{abstract}

\maketitle

\tableofcontents

\section{Introduction}

In~\cite[\S 2.3]{bishop}, C.~Bishop sets the following problem~:

\begin{quote}
``Suppose we stand in a forest with tree trunks of radius $\e > 0$ and no two trees centered closer than unit distance apart. Can the trees be arranged so that we can never see further than some distance $V < \infty$, no matter where we stand and what direction we look in? What is the size of $V$ in terms of $\e$?''
\end{quote}

This is an example of a visibility problem, a topic which has attracted substantial interest over the past decades --- see~\cite{survey} for a survey. Among the problems which gave impetus to research in this field, one can mention the Art Gallery Problem (see~\cite{orourke} for details) or, closer to the spirit of the question set by C.~Bishop, P\'olya's orchard problem. In~\cite[Chap.~5, Problem 239]{polyszego}, G.~P\'olya asks ``how thick must [be] the trunks of the trees in a regularly spaced circular orchard grow if they are to block completely the view from the center''. He then provides a solution in the case that the observer stands at the origin in the plane and that the centres of the trees are the elements of $\Z^2\backslash\{\bm{0}\}$ lying in a disk of integer radius $Q\ge 0$. Allen~\cite{allen} extended this result to the case when the disk has a non--integer radius and Kruskal~\cite{kruskal}  dealt with the situation where the trees are centred at non--zero points of any lattice. Chamizo~\cite{chamizo} also studied an analogue of this problem in hyperbolic spaces and Cusick~\cite{cusick} considered the case when the trees have the shape of any given convex body (Cusick relates this case with the Lonely Runner Conjecture --- see~\cite{chen1, chen2} for further developments).  G.~P\'olya~\cite{polyaforestbis} also took an interest in the visibility in a random and periodic forest, a topic related to the distribution of free path lengths in the Lorentz gas which is still an active domain of research --- see~\cite{markstrombis} and the references therein. On another front, problems of visibility appear in the context of quasi--crystals~\cite{marstro} and of probabilistic billiards in relation with the study of the behaviour of particles~\cite{bachkhmapla}. They are also much studied from an algorithmic point of view and one can therefore find a wealth of literature dealing with them in computer science --- see, e.g., \cite{ghosh, orourke} and the references therein.  Lastly, one should mention that Bishop's question finds its origin in a problem of rectifiability of curves.

\paragraph{} Let $d\ge 1$ be a fixed integer. Define formally a forest $\F$ in $\R^{d+1}$ as a collection of points in $\R^{d+1}$. Given $\epsilon>0$, an $\epsilon$--tree in this forest shall refer to a closed ball centred at an element in $\F$.

\begin{defi}
A set $\F\subset \R^{d+1}$ is a {\em dense forest} if there exists a function $\V~: \epsilon>0 \mapsto \V (\e)\ge 0$ defined in a neighbourhood of the origin such that the following holds for all $\e >0$ small enough~: 
\begin{equation}\label{defdenseforest}
\forall \bm{x}\in\R^{d+1}, \;\;\; \forall \bm{v}\in\Sph^d_2, \;\;\; \exists t\in [0, \V(\e)], \;\;\; \exists \bm{f}\in\F, \;\;\; \left\|\bm{x}+t\bm{v}-\bm{f}\right\|_2\le \e,
\end{equation}
where $\|\,.\,\|_2$ stands for the Euclidean norm in $\R^{d+1}$ and $\Sph^{d}_2$ for the Euclidean sphere in dimension $d+1$. The function $\V$ is then referred to as a {\em visibility function} for $\F$.
\end{defi}

Thus, in a forest with visibility function $\V$, given $\e>0$, {\em any} line segment of length $\V(\e)$ intersect an $\e$--tree (this is the main difference with a P\'olya's orchard--type problem, where one only takes into account those line segments with one of the end points at the origin). It is clear that $\F\subset\R^{d+1}$ is a dense forest whenever the set $\F$ is itself dense. To avoid this pathological case, one may consider at least two types of restrictions for the set $\F$.

On the one hand, one may ask for there to exist a strictly positive real number $r$ such that the gap between any two elements in $\F$ is at least $r$. The forest $\F$ is then said to be {\em uniformly discrete}. This is essentially the condition required by C.~Bishop in the statement of his problem. Y.~Solomon and B.~Weiss~\cite{solweiss} proved the existence of a uniformly discrete dense forest in any dimension. However, their construction is not fully explicit as the forest they obtain is defined as a set of ``visit times'' for the action of a group on a suitable compact metric space. Furthermore, no bound is given for the corresponding visibility function.

On the other hand, one may consider a concept weaker than uniform discreteness, namely that of {\em finite density}. More precisely, a forest $\F$ is said to be of finite density if $$\limsup_{T\rightarrow +\infty}\;\frac{\# \left(\F\cap B_2(\bm{0}, T)\right)}{T^{d+1}}\,<\, +\infty,$$ where $\# X$ stands for the cardinality of a subset $X\subset\R^{d+1}$ and where, given $\bm{x}\in \R^{d+1}$, $B_2(\bm{x}, T)$ denotes the Euclidean ball with radius $T$ centred at $\bm{x}$ (\footnote{We mention in passing that a condition even weaker than $\F$ being of finite density is that of $\F$ being discrete.}). Clearly, a uniformly discrete dense forest has a finite density. Using Dirichlet's Theorem in Diophantine approximation, Y.~Peres~\cite[\S 2.3]{bishop} proved the existence of a dense forest with finite density in the plane (i.e., when $d=1$) whose visibility function $\V$ satisfies the estimate $\V(\e) = O\left(\e^{-4} \right)$ . It is a question left open at end of~\cite{solweiss} to determine whether this bound can be improved. The following, which is the main result of this paper, shows in particular that it is indeed the case.

\begin{thm}\label{mainthmintro}
Let $d\ge 1$. Then there exists a dense forest $\F\subset\R^{d+1}$ with finite density whose visibility function $\V$ satisfies the estimate 
\begin{equation}\label{estimvisimainthm}
\V(\e) \,=\, O_{\eta, d}\left(\e^{-2d-\eta}\right)
\end{equation}
for any $\eta>0$.
\end{thm}

Here, $O_{\eta, d}\left(\: .\:\right)$ means that the implicit constant implied by this notation depends at most on $\eta$ and $d$.

The construction of the dense forest in Theorem~\ref{mainthmintro} is inspired by the work of Y.~Peres and has the merit to be explicit. It will actually be defined as the union of $d+1$ uniformly discrete forests\footnote{As pointed out by a referee, an intermediate condition weaker than uniform discreteness but stronger than finite density is that of finite upper uniform density. The upper uniform density $D^+\left(\F\right)$ of a discrete set $\F\subset\R^{d+1}$ is defined by $$D^+\left(\F\right)\,:=\, \limsup_{T\rightarrow\infty}\:\sup_{\bm{x}\in\R^{d+1}}\:\frac{\# \left(\F\cap B_2\left(\bm{x}, T\right)\right)}{T^{d+1}}\cdotp$$ The condition $D^+\left(\F\right)<\infty$ is then equivalent to $\F$ being the finite union of uniformly discrete sets.}. As a consequence, the argument in~\cite[\S 2.3]{bishop} can easily be modified to show that, given $\e>0$, there exists a uniformly discrete set $\F_\e$ depending on $\e$ satisfying equation~\eqref{estimvisimainthm} for a particular function $\V$.

It should be noted that the visibility problem under consideration is closely related to the Danzer problem, which asks for the existence in $\R^{d+1}$ of a set of finite density intersecting any convex body of volume one. Clearly, such a set is in particular a dense forest with visibility function $\V(\e)=\e^{-d}$. As the Danzer problem is still open, one may consider to weaken the conditions therein in two directions. On the one hand, one may relax the density constraint. Thus, Y.~Solomon and B.~Weiss~\cite{solweiss} showed that there exists a set $\mathcal{D}$ intersecting any convex body with volume one such that $\#\left(\mathcal{D}\cap B_2(\bm{0}, T)\right)\,=\, O\left(T^{d+1}\log T \right)$. On the other hand, one may require less restrictive conditions for the convex bodies under consideration. While John's Theorem (\cite{john}, see also~\cite[Proposition~5.4]{solweiss}) shows that one obtains essentially an equivalent  problem when limiting oneself to the subset of boxes in $\R^{d+1}$, one may relax the volume constraint for these boxes\footnote{A box in $\R^{d+1}$ is the image under an orthogonal transformation of an aligned hyperrectangle of the form $\bm{x}+\prod_{i=1}^{d+1}[a_i, b_i]$, where $a_i< b_i$ ($1\le i\le d+1$) and $\bm{x}\in\R^{d+1}$.}. For obvious reasons, the constructions of dense forests are in this vein. In this respect, the interest of Theorem~\ref{mainthmintro} lies in the fact that it provides the best known bound for the visibility function of a dense forest. This bound is essentially the best one can hope for with the Fourier analytic proof presented in sections~\ref{sec3} and~\ref{sec4} (see also the introduction of section~\ref{sec4} for further details). The construction of the dense forest in section~\ref{sec1} has an independent interest.

\paragraph{\textbf{Notation.}} In addition to those already introduced, the following pieces of
notation shall be used throughout : for any $x\in\R$, let $$\ex (x)\,:=\, \exp\left(2i\pi x\right) \qquad \textrm{ and } \qquad \und\,:=\,\left(1, \, \dots \, , 1\right)\in\R^d.$$ Let $\hat{f}$ denote the Fourier transform of an integrable function $f$ defined in $\R^{d}$. For any $\bm{q}:=\bm{a}/b\in\Q^d$, where $\bm{a}\in\Z^d$, $b\in\N$ and $\gcd(\bm{a}, b)=1$, the denominator $\textrm{den}(\bm{q})$ of $\bm{q}$ shall refer to the well--defined integer $b$. Given $p\in [1; +\infty]$, $\|\, . \, \|_p$ shall stand for the usual $p$--norm in $\R^{d}$. For any two subsets $A, B \subset \R^{d}$, set $$\dist_p(A,B)\, :=\, \inf \left\{\|\bm{a}-\bm{b} \|_p \: :\: \bm{a}\in A, \bm{b}\in B\right\}$$ and let $B_p(\bm{x}, r)$ (resp.~$\Sph^{d-1}_p$) denote the ball with radius $r\ge 0$ centred at $\bm{x}\in\R^d$ (resp.~the unit sphere in $\R^d$) defined with respect to the $p$--norm. Notation such as $\ll_{\alpha, \beta}$, $\gg_{\alpha, \beta}$ and $\asymp_{\alpha, \beta}$ shall mean that the implicit constants implied by these relations depend at most on the parameters $\alpha$ and $\beta$. Also, $\left|Y\right|$ shall refer to the Lebesgue measure of a measurable set $Y\subset\R$ and $\bm{x\cdot y}$ to the usual dot product between two vectors $\bm{x}, \bm{y}\in\R^d$. Given $x,y\in\R$ with $x\le y$, let $$\llbracket x, y\rrbracket\, :=\, \left\{n\in\Z\: :\: x\le n\le y \right\}$$ denote the interval of integers with endpoints $x$ and $y$. Lastly, $\Td$ shall stand for the $d$--dimensional torus $\R^d/\Z^d$ and $\left\|\, . \,\right\|_\infty^{\Td}$ for the corresponding sup--norm; that is, for any $\bm{x}\in\R^d$, $$\left\|\bm{x}\right\|_\infty^{\Td}\, :=\, \inf\left\{\left\|\bm{x}-\bm{m} \right\|_\infty\: : \: \bm{m}\in\Z^d \right\}.$$

\section{Construction of the forest}\label{sec1}

In view of the equivalence of norms in finite dimension and in view of the nature of the result to prove, it should be clear that one may choose to work with any norm instead of the usual Euclidean one in definition~\eqref{defdenseforest}. This choice will indeed only affect the value of the implicit multiplicative constant in~\eqref{estimvisimainthm}. It will be more convenient for us to work with the sup--norm. Thus, definition~\eqref{defdenseforest} shall be modified in the following way~: $$\forall \bm{x}\in\R^{d+1}, \;\;\; \forall \bm{v}\in\Sph_\infty^d, \;\;\; \exists t\in [0, \V(\e)], \;\;\; \exists \bm{f}\in\F, \;\;\; \left\|\bm{x}+t\bm{v}-\bm{f}\right\|_\infty\le \e$$ for any given $\epsilon>0$ small enough. Note that $\Sph_\infty^d$ is the boundary of the hypercube $[-1, 1]^{d+1}$; that is, $$\Sph_\infty^d\, :=\, \left\{\bm{x}:=(x_1, \dots, x_{d+1})\in [-1, 1]^{d+1}\: : \: \exists i_0\in \llbracket 1, d+1\rrbracket, \, \left|x_{i_0}\right|=1 \right\}.$$

Given $j\in \llbracket 1, d+1\rrbracket$, define $$\mathcal{S}_\infty^d(j)\,:=\, \left\{\bm{v}\in\Sph_\infty^d\: :\: \left|v_j\right|=1 \right\}.$$ Fix $\e>0$ and  $V>0$. For $\bm{y}\in\R^{d+1}$ and $\bm{v}\in\mathcal{S}_\infty^d(1)$, let $$I_V(\bm{v}, \bm{y})\, :=\, \left\{\bm{y}+t\bm{v}\: : \: 0\le t \le V\right\}$$ be a line segment.

Assume first that there exists a forest $\F_1\subset \R^{d+1}$ such that for any $\bm{v}\in \mathcal{S}_\infty^d(1)$ and any $\bm{y}\in\R^{d+1}$, 
\begin{equation}\label{distf1}
\dist_\infty\left(\F_1, I_V(\bm{v}, \bm{y}) \right)\le \e.
\end{equation} 
This means that any line segment of Euclidean length at most a constant multiple of $V$ intersects an $\e$--tree in $\F_1$ provided it be supported on a line spanned by a vector $\bv\in\mathcal{S}_\infty^d(1)$. See Figure~\ref{foretasyfigure1} below for an illustration in the plane of the type of forests $\F_1$ that shall be considered later.

\begin{center}
\begin{minipage}{0.4\textwidth}
\makebox[\textwidth][c]{\includegraphics[width=1\textwidth]{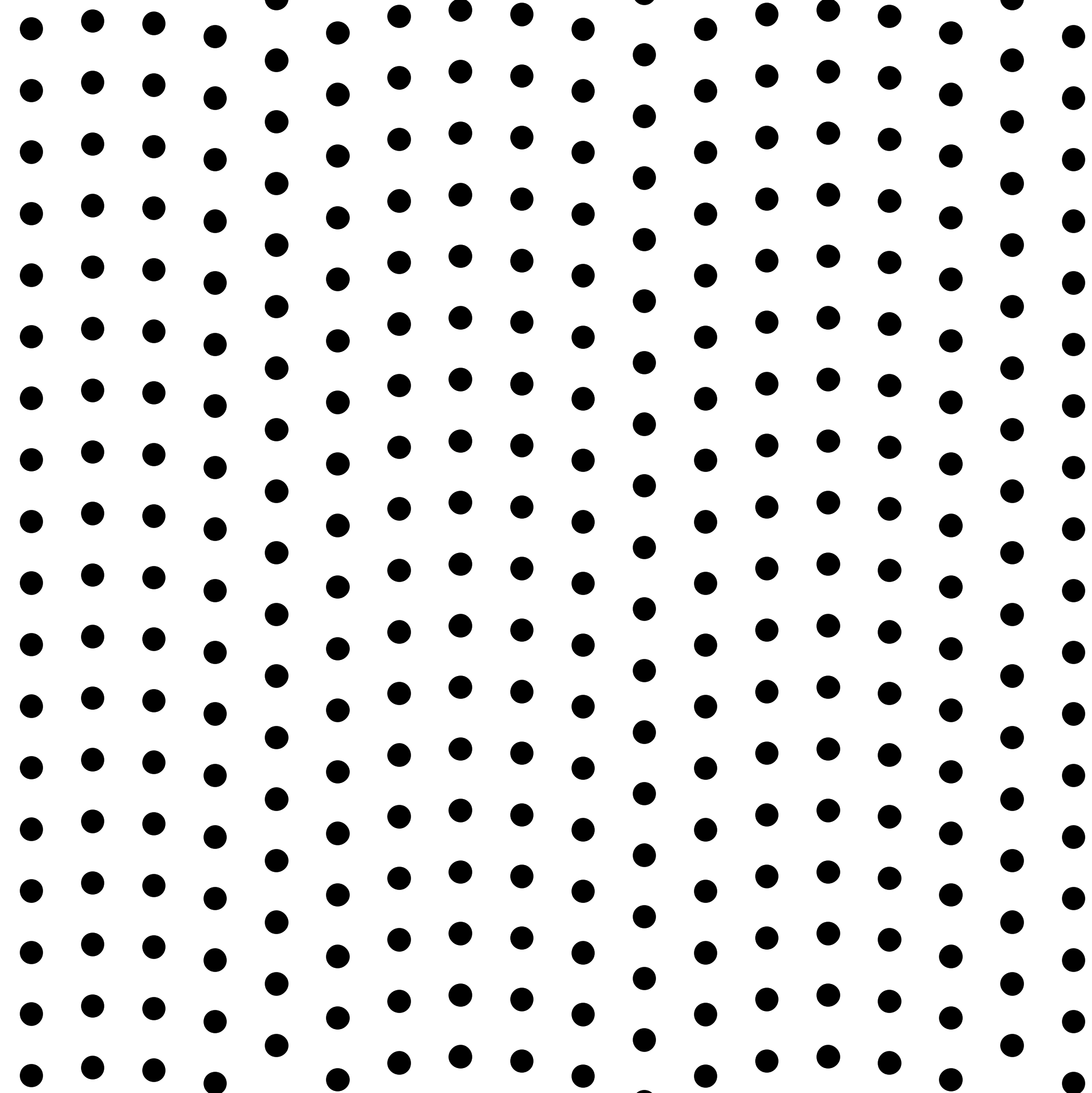}}%
\captionsetup{margin={0pt,0\textwidth}}%
\captionof{figure}{\footnotesize Trees in a forest $\F_1$ which is dense when considering only those line segments supported on a line spanned by a vector $\bv\in\mathcal{S}_\infty^1(1)$.}
\label{foretasyfigure1}
\end{minipage}
\end{center}

Consider now \emph{any} $\bv\in\Sph_\infty^d$. Then, there exists an index $j\in\llbracket 1, d+1\rrbracket$ such that $\bv\in\mathcal{S}_\infty^d(j)$. If $j\neq 1$, let $R_j$ denote the rotation in the $(x_1, x_j)$--plane bringing $x_j$ onto $x_1$ and leaving the orthogonal of this plane invariant. Thus, $R_j(\bv)\in\mathcal{S}_\infty^d(1)$, which implies from~\eqref{distf1} that $$\dist_\infty\left(\F_1\, ,\,  I_V(R_j(\bm{v}), R_j(\bm{y}))\, \right)\le \e$$  for any $\y\in\R^{d+1}$.

In other words, if one denotes by $\F_j$ the forest $\widetilde{R}_j\left(\F_1\right)$ obtained by applying the rotation $\widetilde{R}_j:=R_j^{-1}$ to the forest $\F_1$ ($2\le j \le d+1$), then the visibility $\V(\e)$ in the forest $$\F\, :=\, \bigcup_{j=1}^{d+1}\F_j$$ satisfies the estimate $\V(\e)\le V$ and is in particular finite. If the forest $\F_1$ is furthermore uniformly discrete, then the forest $\F$ has a finite density as the union of a finite number of uniformly discrete forests. 

For the forest $\F_1$ in the plane represented in Figure~\ref{foretasyfigure1}, the resulting forests $\F_2=\widetilde{R}_2\left(\F_1\right)$ and $\F=\F_1\cup\F_2$ are respectively depicted in Figures~\ref{foretasyfigure} and~\ref{foretuniffigure}

\begin{minipage}{0.4\textwidth}
\makebox[\textwidth][c]{\includegraphics[width=1\textwidth]{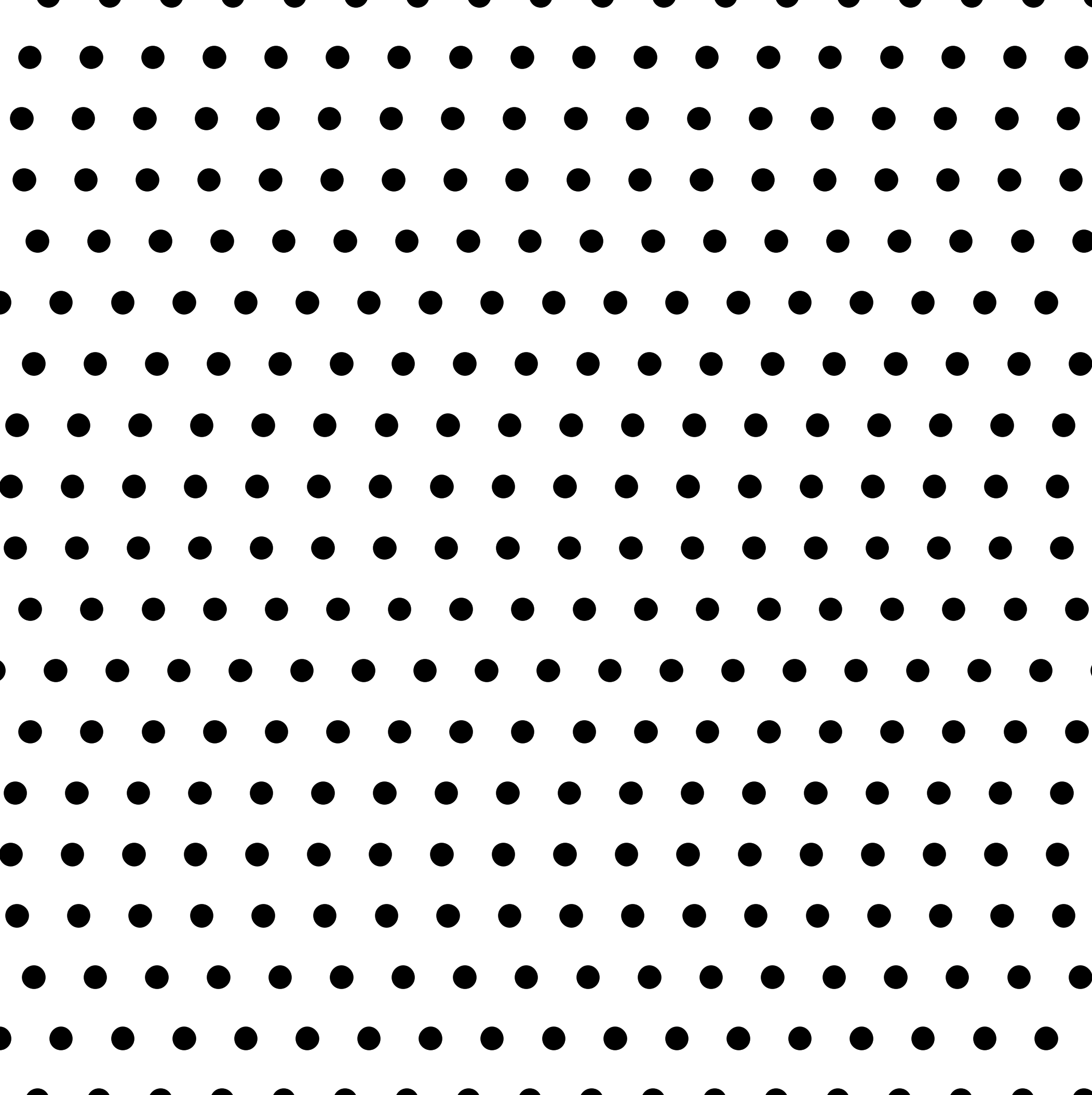}}%
\captionsetup{margin={0pt,0\textwidth}}%
\captionof{figure}{\footnotesize Forest $\F_2$ obtained by rotating the forest $\F_1$ depicted in Figure~\ref{foretasyfigure1} by an angle $\pi/2$. This forest is thus dense  when con\-si\-de\-ring only those line segments supported on a line spanned by a vector $\bv\in\mathcal{S}_\infty^1(2)$.}
\label{foretasyfigure}
\end{minipage}
\hspace{8ex}
\begin{minipage}{0.4\textwidth}
\makebox[\textwidth][c]{\includegraphics[width=1\textwidth]{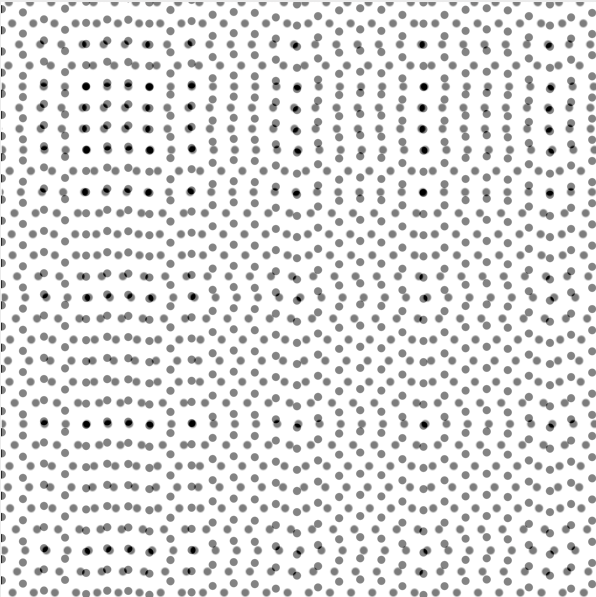}}%
\captionsetup{margin={0pt,0\textwidth}}%
\captionof{figure}{\footnotesize Dense forest $\F$ obtained as the union of the forests $\F_1$ and $\F_2$.\vspace{8ex}}
\label{foretuniffigure}
\end{minipage}

It remains to define a uniformly discrete forest $\F_1$ satisfying condition~\eqref{distf1}. To this end, consider a sequence $\left(a_k\right)_{k\ge 0}$ of real numbers to be specified later. Let $$\F_1\, :=\, \left\{\left(k_1,\, a_{|k_1|}+k_2, \, a_{|k_1|}+k_3, \, \dots, \, a_{|k_1|}+k_d\right)\; : \; (k_1, \, \dots\, ,k_d ) \in\Z^{d+1} \right\}.$$
Clearly, this forest is symmetric with respect to the hyperplane $\left\{x_1=0\right\}$. More explicitly, on each affine hyperplane $\{x_1=k\},$ where $k\in\Z$, the forest is given by 
\begin{equation}\label{interakhyper} 
a_{|k|}\und+\Z^d.
\end{equation}
Assume first that one is able to prove~\eqref{distf1} for all $\bv\in\mathcal{S}_\infty^d(1)$ and $\y\in\R^{d+1}$ such that $I_V(\bm{v}, \bm{y})\subset\left\{x_1\ge 0\right\}$. Consider then any line segment $I_V(\bm{w}, \bm{t})$ determined by the two vectors $\bm{w}\in\mathcal{S}_\infty^d(1)$ and $\bm{t}\in\R^{d+1}$. If $I_V(\bm{w}, \bm{t})\subset\left\{x_1\le 0\right\}$, then the symmetrical of $I_V(\bm{w}, \bm{t})$ with respect to the hyperplane $\left\{x_1=0\right\}$ lies in the half--space $\left\{x_1\ge 0\right\}$ and it is therefore plain from the symmetry of the forest $\F_1$ that~\eqref{distf1} then  holds for  $I_V(\bm{w}, \bm{t})$. If, however, the latter line segment lies on both sides of the hyperplane $\{x_1=0\}$, consider the extended line segment $I_{2V}(\bm{w}, \bm{t})$. Even if it means increasing the value of the implicit constant appearing in the statement of Theorem~\ref{mainthmintro}, it is cleary enough to prove that~\eqref{distf1} holds for $I_{2V}(\bm{w}, \bm{t})$. This is straightforward~: the middle point of $I_{2V}(\bm{w}, \bm{t})$ determines two line segments, one of which lies either in the half--space $\{x_1\ge 0\}$ or in the half--space $\{x_1\le 0\}$. In the former case, \eqref{distf1} holds by assumption and in the latter case this follows from the preceding argument.

All this shows that,  when proving~\eqref{distf1}, it is enough to restrict ourselves to the case that $I_V(\bm{v}, \bm{y})\subset\left\{x_1\ge 0\right\}$, where $\bv\in\mathcal{S}_\infty^d(1)$ and $\y\in\R^{d+1}$. Even if it means defining the line segment $I_V(\bm{v}, \bm{y})$ from its second endpoint (other than $\y$), it may further be assumed without loss of generality that the vector $\bv\in\mathcal{S}_\infty^d(1)$ is given by $\bm{v}:=(1,\, v_2 \, \dots \, , v_{d+1})$. Letting  $(y_1,\, y_2 \, \dots \, , y_{d+1})\in\R^{d+1}$ denote the components of the vector $\y$, these assumptions imply that 
\begin{equation}\label{intersechyper}
\left(\R\bv +\y \right)\cap \left\{x_1=k\right\} \, =\,\left\{\left(k, \: (k-y_1)\bv'+\y'\right) \right\},
\end{equation} 
where $k\in\N$, $\bv':=(v_2 \, \dots \, , v_{d+1})\in [-1, 1]^d$ and $\y':=(y_2 \, \dots \, , y_{d+1})\in\R^{d}$. If one can prove that the distance in the hyperplane $\left\{x_1=k\right\}$ between the point~\eqref{intersechyper} and an element of the set~\eqref{interakhyper} is less than $\e$ for some $k\in\llbracket 0, V\rrbracket$, this would plainly imply that~\eqref{distf1} holds.

The problem has therefore been reduced to find $V>0$ such that the following holds for a given $\e>0$~: 
\begin{equation}\label{butauxiliaire}
\forall \bv'\in [-1, 1]^d, \;\; \;\forall \z\in\R^d, \;\;\; \exists k\in \llbracket 0, V\rrbracket, \;\;\; \left\|k\bv' +\z -a_{k}\und\right\|_{\infty}^{\Td}\, \le \, \e.
\end{equation}

The following lemma essentially shows that it is enough to check~\eqref{butauxiliaire} for a finite number of $\bv'\in [-1, 1]^d$ provided that 
\begin{equation}\label{cntrainteV}
V\, \ge \, \e^{-1}
\end{equation}
which, in view of the statement of Theorem~\ref{mainthmintro}, will be assumed throughout without loss of generality.

\begin{lem}\label{lemreduc}
Assume that~\eqref{cntrainteV} holds for a given $\e>0$. Define 
\begin{align}\label{defD1}
D_1^{(d)}(V)\,:=\, \left\{\q\in [-1,1]^d\cap \Q^d\; : \; \textrm{den}\left(\q\right)\le V^2 \right\}.
\end{align} 
Assume furthermore that 
\begin{equation}\label{conditionderivee}
\forall \q\in D_1^{(d)}(V), \; \;\;\forall \z\in\R^d, \; \;\;\exists k\in \llbracket 0, V\rrbracket, \;\;\; \left\|k\q +\z -a_{k}\und\right\|_{\infty}^{\Td}\, \le \, \e.
\end{equation}

Then, for all $\bv'\in [-1,1]^d$ and all $\z\in\R^d$, there exists $k\in\llbracket 0, V\rrbracket$ such that $$\left\|k\bv' +\z -a_{k}\und\right\|_{\infty}^{\Td}\, \le \, 2\e.$$
\end{lem}

\begin{proof}
Let $\bv'\in[-1, 1]^d$ and $\z\in\R^d$. Clearly, one can find $\q\in D_1^{(d)}(V)$ such that $\left\|\bv' - \q\right\|_{\infty}\le V^{-2}$. Under assumption~\eqref{conditionderivee}, there exist $k\in\llbracket 0, V \rrbracket$ and $\bm{p}\in\Z^d$ satisfying the inequality $\left\| k\q +\z-a_k\und-\bm{p}\right\|_{\infty}\le \e$. It then follows that 
\begin{align*}
\left\| k\bv'+\z-a_k\und-\bm{p}\right\|_{\infty}\, &\le \, k\left\| \bv' -\q\right\|_{\infty} + \left\|k\q+\z-a_k\und-\bm{p} \right\|_\infty\\
&\underset{\eqref{conditionderivee}}{\le} \, \frac{1}{V}+\e\\
&\underset{\eqref{cntrainteV}}{\le} \, 2\e.
\end{align*}
\end{proof}

\section{Reduction of the proof to the estimate of an exponential sum}\label{sec3}

Let $\e\in (0, 1/2)$. Set $B_\e:=[-\e, \, \e]^{d}$ and denote by $\chi_\e$ the characteristic function of the set $B_\e+\Z^d$; that is, 
\begin{equation*}\label{defchieps}
\chi_\e\; :\; \x\in\R^d\;\mapsto\; 
   \left\{
      \begin{aligned}
        1\;&\textrm{if } \|\x\|_{\infty}^{\mathbb{T}^d} \le \e,\\
        0 \;&\textrm{otherwise}.\\
      \end{aligned}
    \right.
\end{equation*}    

Given $V\ge \e^{-1}$, $\q\in D_1^{(d)}(V)$ and $\z\in\R^d$, define 
\begin{equation}\label{defLepsi}
L_\e(V)\,:=\, \sum_{k=V+1}^{2V} \chi_\e\left(k\q+\z-a_k\und\right).
\end{equation} 
Thus, $L_\e(V)>0$ if, and only if, there exists $k\in\llbracket V+1, 2V\rrbracket$ such that $\left\|k\q +\z-a_k\und \right\|_{\infty}^{\mathbb{T}^d} \le \e$. Given $\eta>0$, we shall prove that one can ensure that 
\begin{equation}\label{condL}
L_\e(V)>0\quad \textrm{for some} \quad V\ll_{d,\eta}\e^{-2d-\eta}
\end{equation} 
uniformly in $\q\in D_1^{(d)}(V)$ and $\z\in \R^d$. In view of the reduction operated in~\eqref{butauxiliaire} and in view of Lemma~\ref{lemreduc}, this means that any line segment  with length at most a constant multiple of $\e^{-2d-\eta}$, contained in the half--space $\left\{x_1\ge 0\right\}$ and   supported on a line spanned by a vector in $\mathcal{S}_\infty^d(1)$ hits a $2\e$--tree in the forest $\F_1$. From the discussion held in section~\ref{sec1}, this is clearly enough to establish Theorem~\ref{mainthmintro}.

In order to establish~\eqref{condL}, the characteristic function $\chi_\e$ shall be approximated by a finite trigonometric sum. In dimension one (i.e., when $d=1$ --- the set $B_\e$ then reduces to the interval $[-\e, \e]$), this can be achieved thanks to the well--known Beurling--Selberg extremal functions which approximate the characteristic function of $[-\e, \e]$ by integrable functions with compactly supported Fourier transform. In higher dimensions (i.e., when $d\ge 2$), similar constructions can be found in~\cite{bartmontvaal, cochrane, drmticgy, glyharm} and~\cite[Proof of Theorem~5.25]{glybis} to handle the case of the characteristic function of the hypercube $B_\e$. Here, we adopt the formalism of~\cite{haykelba} and give a simplified version of Theorem~5.1 within which  deals with the general case of aligned parallelotopes. This is achieved in the following proposition, where the majorant function $\psi_\e$ is introduced for the sake of completeness but will not be used thereafter.

\begin{prop}(\cite[Theorem~5.1]{haykelba})\label{thmhaynes}
Given an integer $M\ge 1$, there are trigonometric polynomials $\phi_\e$ and $\psi_\e$ whose Fourier coefficients are supported in $B_\infty\left(\bm{0}, M\right)$ such that for all $\x\in\R^d$, 
\begin{equation}\label{defphipsi}
\phi_\e(\x)\, \le \, \chi_\e(\x)\, \le \, \psi_\e(\x).
\end{equation}
Moreover, there exist two strictly positive constants $C(d)$ and $C'(d)$ for which 
\begin{equation*}
\max\left\{(2\e)^d-\hat{\phi}_\e(\bm{0})\, ; \, \hat{\psi}_\e (\bm{0}) -(2\e)^d\right\}\, \le \, C(d)\cdotp\frac{\e^{d-1}}{M}
\end{equation*}
and 
\begin{equation}\label{estimfouriercoefs}
\max\left\{\left|\hat{\phi}_\e(\bm{m})\right|\, ; \, \left|\hat{\psi}_\e (\bm{m})\right|\right\}\, \le \, C'(d)\cdot r_d(\bm{m})
\end{equation}
for all $\bm{m}=(m_1, \, \dots \, , m_d)\in\Z^d\backslash\left\{\bm{0}\right\}$, where 
\begin{equation}\label{defrd}
r_d(\bm{m}):= \prod_{i=1}^{d}\min\left\{1, \frac{1}{|m_i|}\right\}.
\end{equation}
\end{prop}

By choosing $M:=M_d(\e)$ as
\begin{equation}\label{defM}
M_d(\e)\,:=\, \frac{C(d)\cdot \e^{-1}}{2^d-1},
\end{equation}
we find that
\begin{equation}\label{estimphi0}
\hat{\phi}_\e(\bm{0})\, \ge \, \e^d
\end{equation}
with the notation of Proposition~\ref{thmhaynes}. In view of~\eqref{defLepsi} and~\eqref{defphipsi}, one thus has~:

\begin{align}
L_{\e} (V)\, &\ge \, \sum_{k=V+1}^{2V}\phi_{\e}\left(k\q +\z-a_k\und \right) \nonumber\\
& =\, \sum_{k=V+1}^{2V}\;\sum_{\m\in \llbracket -M_d(\e)\: , \:M_d(\e) \rrbracket^d}\hat{\phi}_\e(\m)\cdot \ex\left(\m\bm{\cdot}\left(k\q +\z-a_k\und\right) \right) \nonumber\\
& \ge\, V\cdot \hat{\phi}_\e(\bm{0}) - \left|\sum_{\m\in \llbracket -M_d(\e)\: , \:M_d(\e) \rrbracket^d\backslash \left\{\bm{0}\right\}} \hat{\phi}_\e (\m) \cdot \sum_{k=V+1}^{2V} \ex\left(\m\bm{\cdot}\left(k\q +\z-a_k\und\right)\right) \right| \nonumber\\
&\ge\, V\cdot \hat{\phi}_\e(\bm{0}) - 2^d\cdot \sum_{\m\in \llbracket 0\: , \: M_d(\e) \rrbracket^d\backslash \left\{\bm{0}\right\}} \left|\hat{\phi}_\e (\m) \right|\cdot \left| \sum_{k=V+1}^{2V} \ex\left(\m\bm{\cdot}\left(k\q +\z-a_k\und\right)\right)\right| \nonumber\\
&\underset{\eqref{estimfouriercoefs} \& \eqref{estimphi0}}{\ge}\, V\cdot \e^d -  2^d\cdot C'(d)\cdot \sum_{\m\in \llbracket 0\: , \: M_d(\e) \rrbracket^d\backslash \left\{\bm{0}\right\}} r_d(\m)\cdot\left| \sum_{k=V+1}^{2V}\ex\left(\m\bm{\cdot}\left(k\q-a_k\und\right)\right)\right|. \label{condistricposi}
\end{align}

Thus, $L_\e(V)>0$ as soon as the right--hand side of~\eqref{condistricposi} is strictly positive. 

The following theorem, which is the core of the proof of Theorem~\ref{mainthmintro}, shall be established in section~\ref{sec4}. The sequence $\left(a_k\right)_{k\ge 0}$ is in particular specified therein.

\begin{thm}\label{mainthm}
Let $K_d>0$ be a constant depending on $d\ge 1$. Then, there exists a real number $\theta$ such that, setting 
\begin{align*}
a_k\, :=\, \theta\cdot k!,
\end{align*}
the following estimate holds for any $\eta>0$ uniformly in $\m\in \llbracket 0\: , \: K_d\cdot V \rrbracket^d\backslash \left\{\bm{0}\right\}$ and $\q\in D_1^{(d)}(V)$~:
\begin{align}\label{expsumbase}
\left| \sum_{k=V+1}^{2V}\ex\left(\m\bm{\cdot}\left(k\q-a_k\und\right)\right)\right|\, \ll_{d, \eta}\, V^{1/2+\eta}.
\end{align}
\end{thm}

In view of~\eqref{cntrainteV} and~\eqref{defM}, $\llbracket 0\: , \: M_d(\e) \rrbracket^d \subset \llbracket 0\: , \: \left(C(d)\cdot V\right)/\left(2^d-1\right) \rrbracket^d$. Thus, from Theorem~\ref{mainthm}, inequality~\eqref{condistricposi} implies that $L_\e(V)>0$ uniformly in $\q\in D_1^{(d)}(V)$ as soon as $$V\cdot \e^d\, \gg_{d, \eta} \,  V^{1/2+\eta}\cdot \left(\sum_{\m\in \llbracket 0\: , \: M_d(\e) \rrbracket^d\backslash \left\{\bm{0}\right\}} r_d(\m) \right) \, \underset{\eqref{defrd}\&\eqref{defM}}{\asymp_{d, \eta}}\,V^{1/2+\eta}\cdot \left|\log \e\right|^d, $$ that is, as soon as $$V\, \gg_{d, \eta}\, \e^{-2d/(1-2\eta)}\cdot\left|\log \e \right|^{2d/(1-2\eta)}.$$ Since this holds for all $\eta>0$, this establishes Theorem~\ref{mainthmintro} modulo the proof of Theorem~\ref{mainthm}.

\section{The proof of Theorem~\ref{mainthm}}\label{sec4}

A classical heuristic argument (cf.~\cite[\S~8.4]{iwakow}) shows that the upper bound in~\eqref{expsumbase} is, up to the factor $\eta$ in the exponent, the best one can hope for. Theorem~\ref{mainthm} actually holds for any sequence $\left(a_k\right)_{k\ge 0}$ of the form
\begin{equation*}
a_k\,:=\, \theta \cdot \alpha_k 
\end{equation*}
for all $k\ge 0$, where $(\alpha_k)_{k\ge 0}$ is  a sequence of integers such that
\begin{equation}\label{defakalphak}
\alpha_k=o\left(\alpha_{k+1}\right).
\end{equation}
 Indeed, such sequences satisfy the following property that shall play a crucial role in the proof  of Theorem~\ref{mainthm}~:

\begin{lem}\label{lemcompatge}
Let $V\ge 1$ and $r\ge 1$ be integers. Assume that the sequence of integers $\left(\alpha_k \right)_{k\ge 0}$ satisfies~\eqref{defakalphak}. Then 
\begin{align*} 
&\Bigg\{\left(l_1, \, \dots\,  , l_r,\,  n_1, \, \dots\, , n_r\right)\in\llbracket V+1, 2V\rrbracket^{2r}\; : \;  \sum_{i=1}^{r}\alpha_{l_i} = \sum_{i=1}^{r}\alpha_{n_i} \Bigg\}\\ 
& \qquad\qquad\qquad= \, \left\{\left(l_1, \, \dots\, , l_r, \, n_1, \, \dots\, , n_r\right)\in\llbracket V+1, 2V\rrbracket^{2r}\; : \; \left\{l_1, \, \dots\, , l_r\right\} = \left\{n_1, \, \dots\, , n_r \right\} \right\}
\end{align*}
for all $V$ large enough (depending on $r$). This implies in particular that $$ \#\left\{\left(l_1, \, \dots\, , l_r,\, n_1, \, \dots\, , n_r\right)\in\llbracket V+1, 2V\rrbracket^{2r}\; : \; \sum_{i=1}^{r}\alpha_{l_i} = \sum_{i=1}^{r}\alpha_{n_i} \right\}\, = \, r!\cdot V^r.$$
\end{lem}

\begin{proof}
Assume that $ \sum_{i=1}^{r}\alpha_{l_i} = \sum_{i=1}^{r}\alpha_{n_i}$ for $\left(l_1, \, \dots\, , l_r, n_1, \, \dots\, , n_r\right)\in \llbracket V+1, 2V\rrbracket^{2r}$. Comparing the leading terms in this equation shows that, under~\eqref{defakalphak}, they must be equal and must appear on both sides with the same multiplicity for $V$ large enough. An easy induction then proves that $ \left\{l_1, \, \dots\, , l_r\right\} = \left\{n_1, \, \dots\, , n_r \right\}$.
\end{proof}

Let $r\ge 1$ and $V\ge K_d^{-1}$ be integers. Fix $\q\in D_1^{(d)}(V)$ and  $\m\in \llbracket 0\: , \: K_d\cdot V \rrbracket^d\backslash \left\{\bm{0}\right\}$. Then 
\begin{align*}
&\int_0^1\left| \sum_{k=V+1}^{2V}\ex\left(\m\bm{\cdot}\left(k\q-\alpha_k \theta \und\right)\right)\right|^{2r}\textrm{d}\theta\, =\\
&\sum_{\underset{V+1\le n_1, \, \dots\, , n_r\le 2V}{V+1\le l_1, \, \dots\, , l_r\le 2V}} \ex\left(\m\bm{\cdot}\q\left(\sum_{i=1}^{r}n_i-\sum_{i=1}^{r}l_i\right)\right)\cdot \int_{0}^{1}\ex\left(\theta\cdot\left(\m\bm{\cdot}\und\right)\cdot\left(\sum_{1=1}^{r}\alpha_{n_i}-\sum_{i=1}^{r} \alpha_{l_i} \right)\right)\textrm{d} \theta.
\end{align*}

Because $\m\bm{\cdot}\und\neq \bm{0}$, by orthogonality, the integral above does not vanish if, and only if, $\sum_{1=1}^{r}\alpha_{n_i}=\sum_{i=1}^{r}\alpha_{l_i}$. From Lemma~\ref{lemcompatge}, this implies that, for $V$ large enough, 

\begin{align*}
&\int_0^1\left| \sum_{k=V+1}^{2V}\ex\left(\m\bm{\cdot}\left(k\q-\alpha_k \theta \und\right)\right)\right|^{2r}\textrm{d}\theta\, =\, r!\cdot V^r.
\end{align*}

Given $\eta>0$, define  now the set $$S_{\m}^{\q}(V)\, := \, \left\{ \theta\in [0, \, 1]\; : \; V^{1/2+\eta}\,\le\, \left|\sum_{k=V+1}^{2V} \ex\left(\m\bm{\cdot}\left(k\q -\alpha_k \theta \und\right)\right) \right|\right\}.$$ It then follows from this definition that 
\begin{align*}
\left|S_{\m}^{\q}(V) \right|\cdot V^{r(1+2\eta)} \, &\le \, \int_{S_{\m}^{\q}(V)} \left| \sum_{k=V+1}^{2V}\ex\left(\m\bm{\cdot}\left(k\q-\alpha_k \theta \und\right)\right)\right|^{2r}\textrm{d}\theta \\
&\le \, \int_0^1 \left| \sum_{k=V+1}^{2V}\ex\left(\m\bm{\cdot}\left(k\q-\alpha_k \theta \und\right)\right)\right|^{2r}\textrm{d}\theta\, = \, r!\cdot V^r,
\end{align*}
whence $$\left|S_{\m}^{\q}(V) \right| \, \le \, r!\cdot V^{-2\eta r}.$$

Note that $$\#  D_1^{(d)}(V)\,\underset{\eqref{defD1}}{\le}\, V^{4d}.$$

Defining $$S(V)\,:=\, \bigcup_{\q\in D_1^{(d)}(V)}\;\bigcup_{\m\in \llbracket 0\: , \: K_d\cdot V \rrbracket^d\backslash \left\{\bm{0}\right\}} S_{\m}^{\q}(V),$$ one thus obtains $$\left|S(V) \right|\, \ll_{r,d}\, V^{4d}\cdot V^{d}\cdot V^{-2r\eta}\, \asymp_{r, d} \, V^{5d-2r\eta}.$$

If the integer $r$ is chosen such that $r>(5d+1)/(2\eta)$, then this inequality implies that the series $\sum_{V\ge 1}\left| S(V)\right|$ converges, hence, from the Borel--Cantelli Lemma, $$\left| \limsup_{V\rightarrow +\infty}\, S(V)\right|\, = \, 0.$$ This means that for almost all $\theta\in [0, \, 1]$, there exists a constant $C(\theta, \eta, d)>0$ depending at most on $\theta, \eta$ and $d$ such that for all $V\ge 1$, for all $\m\in \llbracket 0\: , \: K_d\cdot V \rrbracket^d\backslash \left\{\bm{0}\right\}$ and all $\q\in D_1^{(d)}(V)$, $$\left| \sum_{k=V+1}^{2V}\ex\left(\m\bm{\cdot}\left(k\q-\alpha_k \theta \und\right)\right)\right|\, \le \, C(\theta, \eta, d)\cdot V^{1/2+\eta}.$$ This completes the proof of Theorem~\ref{mainthm}.

\section{Making the construction more explicit}

The construction of the forest $\F$ rests on the choice of the sequence $\left(a_k\right)_{k\ge 0}$ which  in Theorem~\ref{mainthm} depends on a non--explicit parameter $\theta$. In view of the applications of visibility problems in computer science (see, e.g.,  \cite{ghosh} and the references therein), it is desirable to exhi\-bit an explicit sequence with a good estimate for the visibility function of the corresponding fo\-rest. This is the task undertaken in this section, where the result established in Theorem~\ref{thmdernier} below is as good as Theorem~\ref{mainthmintro} under the so--called exponent pair conjecture but is weaker unconditionally.

In order to estimate the exponential sum appearing in~\eqref{expsumbase} (see also~\eqref{condistricposi}), given $\m\in \llbracket 0\: , \:  M_d(\e))\rrbracket^d\backslash\left\{\bm{0}\right\} $ and $\q\in D_1^{(d)}(V)$,  set 
\begin{equation}\label{defpqmsm}
p_{\m, \q}:= \m\bm{\cdot}\q \quad \textrm{ and } \quad s_{\m}:=\m\bm{\cdot}\und\, \ge \, 1.
\end{equation} 
A pair $(\kappa, \lambda)\in [0, 1]^2$ will be said to be an \emph{admissible pair of exponents} if the inequality 
\begin{equation}\label{exppair}
\left|\sum_{k=V+1}^{2V} \ex\left(k p_{\m, \q} - a_k s_{\m}\right)\right|\, \ll \, s_{\m}^\kappa\cdot V^\lambda
\end{equation} 
holds uniformly in $p_{\m, \q}$. Clearly, $(0, 1)$ is an admissible pair of exponents.

\begin{lem}\label{lemexppairvisibilite}
Let $\left(a_k \right)_{k\ge 0}$ be  a sequence of real numbers. Assume that $(\kappa, \lambda)\in [0, 1]\times [0, 1)$ is an admissible pair of exponents for the exponential sum in~\eqref{exppair}. 

Then, the visibility function $\V$ in the forest constructed in section~\ref{sec1} from the sequence $\left(a_k \right)_{k\ge 0}$  satisfies the estimate $$\V(\e)\, = \, O_{\kappa, \lambda, \eta, d}\left(\e^{-(d+\kappa + \eta)/(1-\lambda)} \right)$$ for any $\eta>0$.
\end{lem}

\begin{proof}
As in section~\ref{sec3}, we are looking for a value of $V\ge 1$ depending only on the sequence $\left(a_k\right)_{k\ge 0}$ such that the quantity $L_{\e}(V)$ defined in~\eqref{defLepsi} is strictly positive. From inequality~\eqref{condistricposi}, this happens as soon as  $$V\cdot\e^d \, \gg_d\, \sum_{\m\in \llbracket 0\: , \:  M_d(\e))\rrbracket^d\backslash\left\{\bm{0}\right\}} r_d(\m)\cdot\left| \sum_{k=V+1}^{2V}\ex\left(\m\bm{\cdot}\left(k\q-a_k\und\right)\right)\right|.$$
Under assumption~\eqref{exppair}, this happens whenever  
\begin{equation}\label{estimauxi}
V\cdot\e^d \, \gg_d\, V^{\lambda}\cdot \left(\sum_{\m\in \llbracket 0\: , \:  M_d(\e))\rrbracket^d\backslash\left\{\bm{0}\right\}} r_d(\m)\cdot s_{\m}^{\kappa}\right).
\end{equation}

From the definitions of the quantities $r_d(\m)$ and $s_{\m}$ in~\eqref{defrd} and~\eqref{defpqmsm} respectively, it should be clear that 

\begin{align*}
 \sum_{\m\in \llbracket 0\: , \:  M_d(\e))\rrbracket^d\backslash\left\{\bm{0}\right\}} r_d(\m) \cdot s_{\m}^{\kappa}\, &\ll_d\, \sum_{\m\in \llbracket 1\: , \:  M_d(\e)\rrbracket^d\backslash\left\{\bm{0}\right\}} \frac{\left(\sum_{i=1}^{d}m_1\right)^{\kappa}}{\prod_{i=1}^{d} m_i}\\
&\ll_d\, \sum_{\underset{\left\|\m \right\|_\infty = m_1}{\m \in \llbracket 1, \, M_d(\e)\rrbracket^d}} \frac{m_1^{\kappa}}{\prod_{i=1}^{d} m_i}\\
&\ll_d\, \left(\sum_{m_1=1}^{M_d(\e)} \frac{1}{m_1^{1-\kappa}}\right)\cdot\left( \sum_{m=1}^{M_d(\e)}\frac{1}{m}\right)^{d-1}\\
&\ll_{d, \kappa}\, M_d(\e)^\kappa\cdot\left(\log M_d(\e) \right)^{d-1}.\\
\end{align*}

Thus, for any given $\eta>0$, 

\begin{align*}
 \sum_{\m\in \llbracket 0\: , \:  M_d(\e))\rrbracket^d\backslash\left\{\bm{0}\right\}} r_d(\m) \cdot s_{\m}^{\kappa}\, \ll_{d, \kappa, \eta} \, M_d(\e)^{\kappa +\eta} \, \underset{\eqref{defM}}{\ll_{d, \kappa, \eta}} \, \e^{-(\kappa+\eta)}.
\end{align*}

From~\eqref{estimauxi}, this implies that $L_\e(V)>0$ as soon as $$V^{1-\lambda}\, \gg_{d, \kappa, \eta}\, \e^{-(d+\kappa+\eta)},$$ hence the result.
\end{proof}

In order to find ``good'' exponent pairs, one must first define the sequence $\left(a_k\right)_{k\ge 0}$. Inspired by Van der Corput's method of exponential sums which will be outlined below, it is natural to ask that this sequence should be the leading term in the argument of the exponential appearing in~\eqref{exppair}. However, one should then expect that the faster the sequence $\left(a_k\right)_{k\ge 0}$ grows, the less accurate the estimate of~\eqref{exppair} with this method. Hence, it is natural to set for instance 
\begin{equation}\label{defak}
a_k\, :=\, k(\log k -1)\quad \textrm{for all }\quad k\ge 1
\end{equation} 
and $a_0:=0$.

For the sake of simplicity of notation, let from now on 
\begin{equation*}
p\,:=\,p_{\m, \q} \quad \textrm{ and }\quad s\,:=\,s_{\m}\,\ge\, 1.
\end{equation*} 
Note that from~\eqref{defpqmsm}, $s\ll_d \left\|\m\right\|_\infty\ll_d M_d(\e)$, hence, from~\eqref{cntrainteV} and~\eqref{defM}, 
\begin{equation}\label{sinfV}
s\, \ll_d \, V.
\end{equation} 

 Let $$f(k)\, :=\, pk - s k\left(\log k -1 \right)$$ for all $k\ge 1$. Thus, we are reduced to estimate the exponential sum $\sum_{k=V+1}^{2V}\ex\left(f(k)\right)$. In view of the type of estimate we are looking for (cf.~\eqref{exppair}), we may without loss of generality work with the modified exponential sum 
\begin{equation}\label{defexpf(k)}
\sum_{k=V}^{2V}\ex\left(f(k)\right)
\end{equation} 
($V\ge 1$ ), which will slightly simplify the reasoning below. We first apply to this sum a classical inversion step which is a consequence of the Poisson formula. 

\begin{lem}(\cite[Lemma~3.6]{exppairs})\label{invstep}
Assume that $f$ is a four times continuously differentiable function on an interval $[a, \, b]\subset [V, \, 2V]$ such that $f^{(2)}<0$. Set $\alpha := f'(b)$ and $\beta:=f'(a)$ and assume furthermore that there exists a constant $F>0$ such that the following estimates hold on $[a, \, b]$~: $$\left|f^{(2)} \right|\, \asymp\, FV^{-2}, \quad \left|f^{(3)}\right|\, \ll\, FV^{-3} \quad \textrm{ and }\quad \left|f^{(4)}\right|\, \ll\, FV^{-5}.$$
For any integer $\nu\in [\alpha,\, \beta]$, define the real number $x_\nu\in [a,\, b]$ by $f'(x_\nu)=\nu$.

Then, the following equation holds~: $$\sum_{a\le k \le b}\ex\left(f(k)\right)\, = \, \sum_{\alpha\le \nu \le \beta}\frac{\ex\left(-f(x_\nu)+\nu x_\nu - 1/8\right)}{\left|f^{(2)}(x_\nu)\right|^{1/2}}\, + \, O\left( \log\left(2+FV^{-1} \right)+F^{-1/2}V\right).$$
\end{lem}

It is easily seen that Lemma~\ref{invstep} applies in our case with the choices of the parameters $$a=V, \; \; b=2V,\, \; \; F=sV, \; \; x_\nu = \exp\left(\frac{p-\nu}{s}\right), \; \; \alpha= p-s\cdot \log(2V) \; \; \textrm{ and }\;\; \beta=p-s\cdot \log(V).$$

Upon noting that $$-f(x_\nu)+\nu x_\nu \,= \, -s\cdot \exp\left(\frac{p-\nu}{s}\right),$$ Lemma~\ref{invstep} yields the inequalities~:

\begin{align}
\sum_{k=V}^{2V}\ex\left(f(k)\right)\, &\ll\, \sum_{p-s\cdot\log(2V)\le \nu \le p-s\cdot\log(V)} \frac{\ex\left(-s\cdot \exp\left(\frac{p-\nu}{s} \right)\right)}{\left|f^{(2)}\left(x_\nu\right)\right|^{1/2}} \,+\, \log(2+s)\, + \, \sqrt{\frac{V}{s}} \nonumber \\ 
&\underset{\eqref{sinfV}}{\ll_d}\, \sum_{p-s\cdot\log(2V)\le \nu \le p-s\cdot\log(V)} \frac{\ex\left(-s\cdot \exp\left(\frac{p-\nu}{s} \right)\right)}{\left|f^{(2)}\left(x_\nu\right)\right|^{1/2}} \, + \, \sqrt{V}.\label{expsuminv}
\end{align}

Since $$\left|f^{(2)}\left(x_\nu\right) \right| = \left| s\cdot \exp\left(\frac{\nu-p}{s} \right)\right| \asymp \frac{s}{V}$$ for all $\nu \in [p-s\cdot\log(2V), \, p-s\cdot\log(V)]$, one obtains by partial summation  in the sum on the right--hand side of~\eqref{expsuminv} that

\begin{equation}\label{expsumapresreduc}
\sum_{k=V}^{2V}\ex\left(f(k)\right)\, \ll_d\, \sqrt{V}\, + \, \sqrt{\frac{V}{s}}\max_{p-s\cdot\log(2V)\le j \le p-s\cdot\log(V)} \left|\sum_{p-s\cdot\log(2V)\le \nu \le j}  \ex\left(-s\cdot\exp\left(\frac{\nu-p}{s} \right)\right) \right|.
\end{equation}

In order to estimate the maximum appearing on the right--hand side of this inequality, define 
\begin{equation}\label{defsig}
\sigma\in [0, \, 1/2)
\end{equation}
as the real number such that $s\cdot\log V -p +\sigma\in\Z$. Set $$\gamma\,:=\, -j-s\cdot\log V +p+\sigma \quad \textrm{and}\quad \mu\,:=\, -\nu-s\cdot\log V + p + \sigma.$$
Then, 
\begin{align}
\max_{p-s\cdot\log(2V)\le j \le p-s\cdot\log(V)} & \left|\sum_{p-s\cdot\log(2V)\le \nu \le j}  \ex\left(-s\cdot\exp\left(\frac{\nu-p}{s} \right)\right) \right|\nonumber\\
& \ll\, \max_{1\le \gamma\le s\cdot \log 2}\left| \sum_{\gamma\le \mu\le s\cdot\log 2}e\left(-sV\cdot\exp\left(-\frac{\sigma}{s} \right)\cdot\exp\left(\frac{\mu}{s}\right)\right)\right|. \label{oubli}
\end{align}

Define
\begin{align}\label{defkgamma}
k_\gamma\, :=\, \left\lfloor\frac{1}{\log 2}\cdotp \log\left( \frac{s\cdot \log 2}{\gamma}\right) \right\rfloor\, \ll\, \log s
\end{align}
in such a way that 
\begin{align}
 \max_{1\le \gamma\le s\cdot \log 2}&\left| \sum_{\gamma\le \mu\le s\cdot\log 2}\ex\left(-sV\cdot\exp\left(-\frac{\sigma}{s} \right)\cdot\exp\left(\frac{\mu}{s}\right)\right)\right|\nonumber \\
&\le\,  \max_{1\le \gamma\le s\cdot \log 2}\,\sum_{l=1}^{k_\gamma}\left| \sum_{\frac{s\cdot \log 2}{2^l}\le \mu\le \frac{s\cdot\log 2}{2^{l-1}}}\ex\left(-sV\cdot\exp\left(-\frac{\sigma}{s} \right)\cdot\exp\left(\frac{\mu}{s}\right)\right)\right| \nonumber \\
&\underset{\eqref{defkgamma}}{\ll}\, \log s\cdot \max_{\underset{1\le l \le k_\gamma}{1\le \gamma\le s\cdot \log 2}} \left| \sum_{\frac{s\cdot \log 2}{2^l}\le \mu\le \frac{s\cdot\log 2}{2^{l-1}}}\ex\left(-sV\cdot\exp\left(-\frac{\sigma}{s} \right)\cdot\exp\left(\frac{\mu}{s}\right)\right)\right|. \label{inegdeplus}
\end{align}

From~\eqref{defsig}, it should be clear that for any $\gamma\in \llbracket 1, \, s\,\cdot \log 2\rrbracket$, $l\in\llbracket 1, \, k_\gamma \rrbracket$ and $r\ge 1$, $$\left|\frac{\textrm{d}^r}{\textrm{d}\mu^r}\left(-sV\cdot\exp\left(-\frac{\sigma}{s} \right)\cdot\exp\left(\frac{\mu}{s}\right) \right) \right|\, \asymp\, \frac{V}{s^{r-1}}$$ uniformly in $\mu\in [s\cdot\log 2/2^l, \, s\cdot\log 2/2^{l-1}]$, where the implicit constant is absolute. One may therefore apply the theory of exponent pairs in order to estimate the exponential sum appearing on the right--hand side of~\eqref{inegdeplus} (see, e.g., \cite[p.44]{ivic} for a justification of this claim). For a detailed exposition on this topic, the reader is referred to~\cite{exppairs}. We recall hereafter the facts which will be useful to our purpose.

 The inversion step (i.e.~Lemma~\ref{invstep}) has reversed the role played by the integers $s$ and $V$ in the initial sum~\eqref{defexpf(k)}~: in~\eqref{inegdeplus}, the integer $s$ now determines the range of the sum and $V$ is considered as a parameter. As a consequence, an admissible exponent pair for the latter sum takes the form of a pair $\left(\kappa', \, \lambda' \right)\in [0, 1]^2$ satisfying the inequality
\begin{align}\label{newexppairs}
\left| \sum_{\frac{s\cdot \log 2}{2^l}\le \mu\le \frac{s\cdot\log 2}{2^{l-1}}}\ex\left(-sV\cdot\exp\left(-\frac{\sigma}{s} \right)\cdot\exp\left(\frac{\mu}{s}\right)\right)\right|\, \ll\, V^{\kappa'}\cdot s^{\lambda'}.
\end{align}

Thus, if $\left(\kappa', \, \lambda' \right)$ is an admissible exponent pair in~\eqref{newexppairs} with $\lambda'\ge 1/2$, in view of~\eqref{expsumapresreduc} and~\eqref{oubli}, 
\begin{equation} \label{relexppair}
\left(\kappa, \, \lambda \right):=\left(\lambda'-1/2, \, \kappa'+1/2 \right)
\end{equation} 
is an admissible pair for~\eqref{defexpf(k)} (in~\eqref{expsumapresreduc}, the second term of the sum on the right--hand side is bounded above by a constant multiple of $V^\lambda \cdot s^\kappa$ which dominates the term $\sqrt{V}$ because $\kappa\ge 0$ and $\lambda\ge 1/2$).

Let $\eta>0$. According to the exponent pair conjecture, $\left(\kappa', \, \lambda' \right)=\left(\eta, 1/2+\eta\right)$ is an admissible exponent pair in~\eqref{newexppairs}, in which case $\left(\kappa, \, \lambda\right)=\left(\eta, 1/2+\eta\right)$ is also an admissible exponent pair in~\eqref{defexpf(k)}. From Lemma~\ref{lemexppairvisibilite}, this implies that the visibility function $\V$ in the forest constructed from the sequence defined in~\eqref{defak} is such that $$\V(\e)\, = \, O_{\eta, d}\left(\e^{-2d-\eta} \right).$$ This is essentially the result stated in Theorem~\ref{mainthmintro}.

In order to obtain an unconditional result, one may successively repeat, following Van der Corput's idea, a so--called $A$--process and a so--called $B$--process when treating the exponential sum in~\eqref{newexppairs}. The $A$--process consists of an application of the Weyl--Van der Corput Lemma (cf.~\cite[Lemma~2.5 ]{exppairs}) and transforms an exponent pair $(k, \, l)$ into a pair $(k/(2k+2), \, (k+l+1)/(2k+2))$. The $B$--process is essentially the inversion step presented in Lemma~\ref{invstep} above and transforms an exponent pair $(k, \, l)$ into a pair $(l-1/2, \, k+1/2)$. As the latter step is involutive (i.e.~$B^2$ gives back the initial sum), any application of Van der Corput's method takes the form 
\begin{equation}\label{vdcalogo}
(B)A^{r_1}BA^{r_2}\dots BA^{r_n}B \quad (r_i\ge 0, \, 1\le i \le n),
\end{equation} 
where $A^r$ means that the $A$--process is applied $r\ge 0$ times. Here, one must start with any given pair of exponents on the right--hand side of~\eqref{vdcalogo} (e.g., the trivial one $(0,1)$) to obtain a new pair of exponents after the successive applications of the $A$ and $B$ processes as indicated in~\eqref{vdcalogo}. The choice of the parameters $r_i$ represents a difficult problem of optimisation which was rigorously formalised by Van der Corput in the twenties and then simplified by Philips (1933). More recently, Graham and Kolesnik~\cite[Chap.~5]{exppairs} proposed an algorithm to this end.

In our case, in view of Lemma~\ref{lemexppairvisibilite} and of~\eqref{relexppair}, the goal is to find optimal exponent pairs $\left(\kappa', \, \lambda' \right)$ so as to minimise the quantity 
\begin{equation}\label{qteamin}
\frac{\lambda' - 1/2+d}{1/2-\kappa'}\,= \, \frac{\kappa+d}{1-\lambda}\cdotp
\end{equation} 
The algorithm provided in~\cite[Chap.~5]{exppairs} yields that the pair of exponents 
\begin{equation}\label{pairretenue}
\left(\kappa', \, \lambda' \right)\, = \,(1/42, \, 25/28)
\end{equation}
 resulting from the operations $A^3BA^2B (0, \, 1)$ is an admissible pair for~\eqref{newexppairs} which approaches the infimum of~\eqref{qteamin}. As the algorithm is infinite and as further iterations do not seem to provide substantial improvement (the value of~\eqref{qteamin} decreases by less than 0.001 when $d=1$ for the next 20 iterations), we prefer to state below the unconditional result obtained with the pair~\eqref{pairretenue}. The interested reader may try to improve on this bound by delving further into the algorithm detailed in~\cite[Chap.5]{exppairs}.  

Thus, from Lemma~\ref{lemexppairvisibilite}, the visibility function $\V$ in the forest under consideration satisfies the estimate $\V(\e)\, =\, O_{\eta, d}\left(\e^{-(d+11/28)/(10/21)-\eta}\right)$ for any $\eta>0$. Note that in the plane (i.e.~when $d=1$), this gives a visibility function such that $\V(\e)\, =\, O_{\eta}\left(\e^{-2.925-\eta}\right)$ instead of $\V(\e)\, =\, O_{\eta}\left(\e^{-2-\eta}\right)$ ($\eta>0$) under the exponent pair conjecture. This is still much better than the bound $O\left(\e^{-4}\right)$ obtained by Y.~Peres in~\cite[\S 2.3]{bishop}.

The results proved in this section are summarised in the theorem below.

\begin{thm}\label{thmdernier}
Let $\eta>0$ and let $\left(a_k\right)_{k\ge 0}$ be the sequence defined in~\eqref{defak}.

Then, the visibility function $\V$ for the forest $\F$ constructed in section~\ref{sec1} from this sequence satisfies the following estimates~:
\begin{itemize}
\item under the exponent pair conjecture, $$\V(\e)\, = \, O_{\eta, d}\left(\ \e^{-2d-\eta}\right) \,;$$
\item unconditionally,  $$\V(\e)\, = \, O_{\eta, d}\left(\ \e^{-2.1\cdot d-0.825-\eta}\right).$$
\end{itemize}
\end{thm}

\end{document}